\documentclass[11 pt,onecolumn]{IEEEtran}

\IEEEoverridecommandlockouts     

\usepackage{graphicx,algorithm,algorithmic,subfigure,latexsym,amsmath,amsfonts,amssymb,cite,mathrsfs}
\usepackage{epsfig,color}
\usepackage{psfrag}
\usepackage{pstricks}

\newtheorem{lemma}{Lemma}[section]
\newtheorem{theorem}{Theorem}[section]
\newtheorem{corollary}{Corollary}[section]
\newtheorem{remark}{Remark}[section]
\newtheorem{definition}{Definition}[section]
\newtheorem{example}{Example}[section]

\newtheorem{assumption}{Assumption}[section]

\def\rrr#1\\{\par
\medskip\hbox{\vbox{\parindent=2em\hsize=6.12in
\hangindent=4em\hangafter=1#1}}}
\allowdisplaybreaks

\begin{document}

\title{\huge Semistability of Nonlinear Systems Having a Continuum of Equilibria and Time-Varying Delays}

\author{Qing~Hui\\
{\small{Control Science and Engineering Laboratory\\
Department of Mechanical Engineering \\
Texas Tech University \\
Lubbock, TX 79409-1021\\
Technical Report CSEL-01-10, January 2010}}
\thanks{This work was supported by the Defense Threat Reduction Agency, Basic Research Award \#HDTRA1-10-1-0090, to Texas Tech University.}
}

\markboth{Technical Report CSEL-01-10}{}

\maketitle
\IEEEpeerreviewmaketitle
\pagestyle{empty}
\thispagestyle{empty}
\baselineskip 22pt


\begin{abstract}
In this report, we develop a semistability analysis framework for nonlinear systems with bounded time-varying delays with applications to stability analysis of multiagent dynamic networks with consensus protocols in the presence of unknown heterogeneous time-varying delays along the communication links. We show that for such a nonlinear system having nonisolated equilibria, if the system asymptotically converges to a constant time-delay system and this new system is semistable, then the original time-varying delay system is semistable. In proving our results, we extend the limiting differential equation approach to time-varying delay systems and also develop some new convergence results for functional differential equations.
\end{abstract}


\section{Introduction}

Delays are unavoidable in communication, where information has to be transmitted over a physical distance. Unfortunately, very little research has been done to investigate the effect of delays on stability of consensus of multiagent networks. To accurately describe
the evolution of networked cooperative systems, it is necessary to include in
any mathematical model of the system dynamics some information about
the past system states. In this case, the state of the system at a
given time involves a portion of trajectories in the space of
continuous functions defined on an interval of the state space, which leads to (infinite-dimensional) delay dynamical systems
\cite{HV:1993}.

Previously, most of the reported work has either explicitly or implicitly employed the assumption that delays are known and continuously differentiable. Under such an assumption, one can use the delayed state of an agent in its own local control law to match the delays of the states from the neighboring agents \cite{BT:CDC:2005,LJL:SCL:2008,SM:TAC:2004,SWX:CDC:2006}, i.e. agent $i$ can use a delayed version of its own state, $x_i(t-\tau_{ij}(t))$. Under that assumption, the control law is
\begin{eqnarray}\label{eq:cl}
u_i = \sum_{j \in \mathcal{N}_i} a_{ij}(x_j(t-\tau_{ij})-x_i(t-\tau_{ij})),
\end{eqnarray} where $\mathcal{N}_{i}$ denotes the set of all other agents having a communication with agent $i$ and $a_{ij}\in\mathbb{R}$.
If the delays are constant and uniform, $\tau_{ij} = \tau$ for all $ij$, then the network dynamics are of the form of time-delayed linear systems with the system matrix being the Laplacian, $\dot x = Lx(t-\tau)$, for which various analysis tools for linear systems with delays can be applied \cite{LJL:SCL:2008,SM:TAC:2004,SWX:CDC:2006}. Additionally, the control law in (\ref{eq:cl}) allows one to utilize disagreement dynamics, in which the disagreement $x_j(t-\tau_{ij})-x_i(t-\tau_{ij})$ is the delayed version of the disagreement $x_j(t)-x_i(t)$. Because of the preceding property, one can study the behavior of the networks using disagreement dynamics or reduced disagreement dynamics in a similar fashion to the case without delays (the reduced disagreement dynamics are asymptotically stable). However, if the delays are unknown, time-varying, and not uniform over the communication links, the assumption that agent $i$ has access to the delayed state $x_i(t-\tau_{ij}(t))$ raises a practical concern. If agent $i$ does not have $x_i(t-\tau_{ij}(t))$ to use in the control protocol (in which case we say that the delays are asymmetric), the control law actually becomes
\begin{eqnarray}\label{eq:cl2}
u_i = \sum_{j \in \mathcal{N}_i} a_{ij}(x_j(t-\tau_{ij})-x_i(t)).
\end{eqnarray}
Because $x_j(t-\tau_{ij})-x_i(t)$ is no longer the delayed version of the disagreement $x_j(t)-x_i(t)$, the derivatives of the disagreements are not functions of the disagreements only), and hence, the approaches in \cite{LJL:SCL:2008,SM:TAC:2004,SWX:CDC:2006} are not applicable to networks with the protocol (\ref{eq:cl2}).
Stability of dynamic networks in such a situation has only recently been addressed \cite{CHHR:SCL:2008,PJ:CDC:2006,TL:TAC:2008}, most of which are limited to the case of constant time delays. In particular, the authors in \cite{CHHR:SCL:2008} have shown that dynamic networks with consensus protocols in the presence of heterogeneous delays are stable for arbitrary constant delays. Another closely related work is \cite{XW:TAC:2008}, where the authors consider networks with different arrival times for communication and with zero-order hold control laws, which leads to discrete-time dynamic networks formulation without time-delays for the overall closed loop. Left open is the problem of stability and convergence of time-varying consensus dynamic networks in the presence of unknown asymmetric non-uniform time-varying delays, which turns out to be a consequence of the more general results in this report.

In this report, we develop a general framework for semistability analysis of nonlinear systems having nonisolated equilibria and bounded time-varying delays in which the delays are unknown and continuous with respect to time, not necessarily continuously differentiable. Here semistability is the property
whereby every trajectory that starts in a neighborhood of a Lyapunov
stable equilibrium converges to a (possibly different) Lyapunov
stable equilibrium. The basic assumption for the main result in this report involves the idea of \textit{limiting equations} \cite{Artstein:JDE:1977} by assuming that the original time-varying delay system asymptotically converges to an autonomous system with constant delays. Using these results, next we present stability analysis of time-varying consensus dynamic networks in the presence of unknown asymmetric non-uniform time-varying delays. The main feature of the proposed framework is that the assumption on continuous differentiability of time delays is considerably weakened by use of a limiting function assumption, which is more natural and useful in many systems, particularly for dynamical systems with \textit{autonomous self-regulating time lags} \cite{BC:PNAS:1959,Cooke:JDE:1970,Rek:DAN:1958,Ginzburg:DU:1970,Ruiz:JDE:1976,Stephan:JDE:1969,Stephan:JDE:1970}.  

\section{Mathematical Preliminaries}\label{TDLMI}

Let $\mathbb{R}^{n}$ denote the real Euclidean space of $n$-dimensional column vectors and let $|x|$ denote the norm of the vector $x$ in $\mathbb{R}^{n}$. Let $r\geq 0$ be a constant and let $\mathcal{C}([-r,0],\mathbb{R}^{n})$ denote the space of continuous functions that map the interval $[-r,0]$ into $\mathbb{R}^{n}$ with the topology of uniform convergence and designate the norm of an element $\phi\in\mathcal{C}$ by $\|\phi\|=\sup_{-r\leq\theta\leq0}|\phi(\theta)|$. Let $\rho\geq0$ and $x\in\mathcal{C}([-r,\rho],\mathbb{R}^{n})$. Then for any $t\in[0,\rho]$, define $x_{t}\in\mathcal{C}$ by $x(t+s)$, $s\in[-r,0]$.

Consider the nonlinear dynamical system with bounded time-varying delays given by the \textit{differential difference equation} \cite{BC:1963}
\begin{eqnarray}\label{DS}
\dot{x}(t)=f(x(t))+g(x(t-\tau_{1}(t)),\ldots,x(t-\tau_{m}(t))),
\end{eqnarray} where $x(t)\in\mathbb{R}^{n}$, $f:\mathbb{R}^{n}\to\mathbb{R}^{n}$ is globally Lipschitz continuous, $g:\mathbb{R}^{n}\times\cdots\times\mathbb{R}^{n}\to\mathbb{R}^{n}$ is globally Lipschitz continuous, and $\tau_{k}:\mathbb{R}\to\mathbb{R}$ is continuous but \textit{not} necessarily differentiable, $k=1,\ldots,m$. Throughout this report, we make the following standing assumptions on (\ref{DS}).

\begin{assumption}\label{A2}
For every $k=1,\ldots,m$ and $t\in\mathbb{R}$, $0\leq\tau_{k}(t)\leq h$, where $h>0$ is a constant.
\end{assumption}

\begin{assumption}\label{A1}
The equilibrium set $\mathcal{E}:=\{x_{t}\in\mathcal{C}=\mathcal{C}([-h,\infty),\mathbb{R}^{n}):x(t+s)\equiv\alpha\in\mathbb{R}^{n},\,\,{\rm{for}}\,\,{\rm{all}}\,\,t\geq0,\,\,{\rm{for}}\,\,{\rm{all}} \,\,s\in[-h,0],\,\,{\rm{and}}\,\,f(\alpha)+g(\alpha,\ldots,\alpha)=0\}$ is connected.
\end{assumption}

Recall that a set $\mathcal{E}\subseteq\mathcal{C}$ is \textit{connected} if every pair of open sets $\mathcal{U}_{i}\subseteq\mathcal{C}$, $i=1,2$, satisfying $\mathcal{E}\subseteq\mathcal{U}_{1}\cup\mathcal{U}_{2}$ and $\mathcal{U}_{i}\cap\mathcal{E}\neq\varnothing$, $i=1,2$, has a nonempty intersection. Assumption~\ref{A2} implies that the differential difference equation (\ref{DS}) has bounded time-varying delays. Assumption~\ref{A1} implies that the equilibria of (\ref{DS}) are \textit{nonisolated} equilibrium points. This situation occurs in many practical problems such as compartmental modeling of biological systems \cite{HCH:2010}, thermodynamic systems \cite{HCN:2005}, multiagent coordinated networks \cite{CHHR:SCL:2008,SM:TAC:2004,VHM:ACC:2010}, and synchronization of coupled oscillators \cite{PJ:CDC:2006}. 

\begin{example}
Consider a special case of (\ref{DS}) where $f(x)=Ex$, $g(x,\ldots,x)=\sum_{k=1}^{m}F_{k}x$, and $E,F_{k}\in\mathbb{R}^{n\times n}$ are matrices, $k=1,\ldots,m$. If $E+\sum_{k=1}^{m}F_{k}$ is singular, then $\mathcal{E}$ is a connected set, i.e., (\ref{DS}) has nonisolated equilibria. A relevant example for this case is the consensus problem with time delays \cite{CHHR:SCL:2008,SM:TAC:2004,VHM:ACC:2010} given by the consensus protocol
\begin{eqnarray}\label{eq:sys1}
\dot{x}(t)=Ex(t)+\sum_{k=1}^{m}F_{k}x(t-\tau_{k}(t)),
\end{eqnarray}
where $E+\sum_{k=1}^{m}F_{k}$ is a Laplacian. 
\end{example}

Given $\phi\in\mathcal{C}$ and $\tau>0$, a function $x(\phi)(\cdot)$ is called a \textit{solution} to (\ref{DS}) on $[-h,\tau)$ with initial condition $\phi$ if $x\in\mathcal{C}([-h,\tau),\mathbb{R}^{n})$, $x_{t}\in\mathcal{C}$, $x(t)$ satisfies (\ref{DS})
for every $t\in[0,\tau)$, and $x(\phi)(0)=\phi$ \cite{HV:1993}. We use the short notation $x_{t}(\phi)$ for $x(\phi)(t)$. Recall that a point $y\in\mathcal{C}$ is a \textit{positive limit point} of a solution $x_{t}(\phi)$ to (\ref{DS}) with $x_{0}(\phi)=\phi$, if there exists a nonnegative sequence $\{t_{n}\}_{n=1}^{\infty}$ with $t_{n}\to+\infty$ as $n\to\infty$ such that $\lim_{n\to\infty}x_{t_{n}}(\phi)=y$. The set of all such positive limit points, denoted by $\omega(\phi)$, is called the \textit{positive limit set} of $x_{0}(\phi)=\phi\in\mathcal{C}$ \cite[p.~102]{HV:1993}. For the notions of \textit{bounded solutions} and \textit{invariant sets}, see Definition 1.2 in \cite[p.~131]{HV:1993} and Definition 2.2 in \cite[p.~104]{HV:1993}, respectively. Finally, recall that the equilibrium solution $x_{t}(\phi)\equiv x_{\rm{e}}$ of (\ref{DS}) is \textit{Lyapunov stable} relative to $\mathcal{D}$ if for every $\varepsilon>0$, there exists $\delta=\delta(\varepsilon)>0$ such that $\phi\in\mathcal{B}_{\delta}(x_{\rm{e}})\cap\mathcal{D}$ implies $x_{t}(\phi)\in\mathcal{B}_{\varepsilon}(x_{\rm{e}})\cap\mathcal{D}$ for all $t\geq0$, where $\mathcal{B}_r(s)$ denotes the open ball centered at $s$ with radius $r$. Motivated by Proposition 3.1 of \cite{HHB:TAC:2009} and Lemma 3.1 of \cite{HKTW:JDE:1994}, we have the following result.

\begin{lemma}\label{omega}
Let $\mathcal{D}\subseteq\mathcal{C}$ be invariant with respect to (\ref{DS}). Assume that the solutions of (\ref{DS}) are bounded and let $x_{t}(\phi)$ be a
solution of (\ref{DS}) with $x_{0}(\phi)=\phi\in\mathcal{D}$. If
$z\in\omega(\phi)$ is a Lyapunov stable equilibrium solution to (\ref{DS}) relative to $\mathcal{D}$, then
$z=\lim_{t\to\infty}x_{t}(\phi)$ and $\omega(\phi)=\{z\}$.
\end{lemma}

\begin{proof} Since the solutions of (\ref{DS}) are bounded, it follows from Lemma 1.3 of \cite[p.~103]{HV:1993} that $\omega(\phi)$ is nonempty. Next, note that it follows from \cite[p.~101]{HV:1993} that the solution to (\ref{DS}) is the \textit{process generated by the retarded functional differential equation} (\ref{DS}). Now the proof of the result is similar to the proofs
of Proposition 3.1 of \cite{HHB:TAC:2009} and Lemma 3.1 of \cite{HKTW:JDE:1994}. 
\end{proof}

\begin{definition}
Let $\mathcal{D}\subseteq\mathcal{C}$ be invariant with respect to (\ref{DS}). An equilibrium solution $x_{t}(\phi)\equiv x_{\rm{e}}\in\mathcal{E}\cap\mathcal{D}$ of (\ref{DS}) is \textit{semistable} relative to $\mathcal{D}$ if there exists a set $\mathcal{U}\subseteq\mathcal{D}$ containing $x_{\rm{e}}$ such that $\mathcal{U}=\mathcal{S}\cap\mathcal{D}$ for some open set $\mathcal{S}\subseteq\mathcal{C}$, every equilibrium solution $x_{t}(\phi)\equiv x_{\rm{e}}$ in $\mathcal{U}$ is Lyapunov stable relative to $\mathcal{D}$, and for every initial condition $\phi\in\mathcal{U}$, $\lim_{t\to\infty}x_{t}(\phi)$ exists. The system (\ref{DS}) is \textit{semistable} relative to $\mathcal{D}$ if for every $x_{\rm{e}}\in\mathcal{E}\cap\mathcal{D}$, $x_{t}(\phi)\equiv x_{\rm{e}}$ is semistable relative to $\mathcal{D}$.
\end{definition}

\section{Main Results}

\subsection{General Results for Nonlinear Time Delay Systems}

In this section, we propose a \textit{limiting delay system} approach to study the asymptotic behavior of (\ref{DS}). Specifically, it follows from (\ref{DS}) that
$\dot{x}(t)=f(x(t))+g(x(t-h_{1}),\ldots,x(t-h_{m}))+g(x(t-\tau_{1}(t)),\ldots,x(t-\tau_{m}(t)))-g(x(t-h_{1}),\ldots,x(t-h_{m}))$, where $0\leq h_{k}\leq h$, $k=1,\ldots,m$, are some constants that are \textit{not} necessarily known. Next, define
$\mathcal{X}(t):=g(x(t-\tau_{1}(t)),\ldots,x(t-\tau_{m}(t)))-g(x(t-h_{1}),\ldots,x(t-h_{m}))$. Then we have
$\dot{x}(t)=f(x(t))+g(x(t-h_{1}),\ldots,x(t-h_{m}))+\mathcal{X}(t)$.
Note that if $x_{t}(\phi)\equiv\alpha\in\mathcal{E}$, then $\mathcal{X}(t)=0$.

\begin{definition}\label{def}
Let $\mathcal{D}\subseteq\mathcal{C}$ be invariant with respect to (\ref{DS}). If for every initial condition $x_{0}(\phi)=\phi\in\mathcal{D}$, the solution $x_{t}(\phi)\in\mathcal{D}$ to (\ref{DS}) satisfies that there exists a sequence $\{t_{n}\}_{n=1}^{\infty}$ with $t_{n}\to+\infty$ as $n\to\infty$ such that $\lim_{n\to\infty}\mathcal{X}(t+t_{n})=0$ uniformly in $t$ on every compact subset of $[0,\infty)$,
then the system
\begin{eqnarray}\label{limiting}
\dot{z}(t)=f(z(t))+g(z(t-h_{1}),\ldots,z(t-h_{m}))
\end{eqnarray} with the initial condition $z_{0}(\phi)=\phi\in\mathcal{D}$ is called a \textit{limiting delay system} of (\ref{DS}) relative to $\mathcal{D}$. 
\end{definition}

\begin{remark}
Definition~\ref{def} can be generalized to the case where the limiting delay system (\ref{limiting}) is of the form (\ref{DS}), i.e., the time-delays are not necessarily constant. Many results developed in this report can be parallel extended to this case. However, in this case, one may not have some simple criteria to test semistability of the limiting delay system with time-varying delays. 
\end{remark}

The idea of the limiting delay system approach is inspired by the limiting equation approach originated from \cite{Artstein:JDE:1977} and being extended to various finite- and infinite-dimensional dynamical systems by changing the definition of limiting functions \cite{MK:PM:1987,LLC:TAC:2001,LR:JDE:2003,AP:MN:2000,Sedova:JMAA:2003}. Our definition extends this limiting equation approach to differential difference equations and gives a new definition of limiting systems for time-delay systems.

It is important to note that the proposed limiting delay system approach is closely related to the problem of the effect of delay on differential difference equations studied in \cite{BC:PNAS:1959,Cooke:JDE:1970,Rek:DAN:1958,Ginzburg:DU:1970,Ruiz:JDE:1976,Stephan:JDE:1969,Stephan:JDE:1970}. Specifically, for a linear differential difference equation given by $\dot{x}(t)=Ax(t)+Bx(t-\tau(t))$, it is of interest to determine the asymptotic behavior of the solutions if $\tau(t)$ is ``close" to a constant for large $t$. In this case, the Banach space $\mathcal{C}$ is too large in general (this is one of the reasons why $\mathcal{D}$ is used in Definition~\ref{def}; another reason is that for some practical systems such as \textit{nonnegative systems} \cite{HCH:2010} in which the state variables are always nonnegative, it would be more appropriate to discuss the dynamic behavior of such systems under these constraints). Cooke \cite{Cooke:JDE:1970} has shown that the \textit{Sobolev space} $W^{1,\infty}$ \cite{Evans:1998} can be effectively used to discuss this problem (see also \cite{BC:PNAS:1959}). The Sobolev space $W^{1,\infty}$ or $W^{1,p}$ for some $p\in(1,\infty)$ has been used for the existence of periodic solutions \cite{Ginzburg:DU:1970,Ruiz:JDE:1976,Stephan:JDE:1969,Stephan:JDE:1970}. However, to our best knowledge, there are no general \textit{nonlinear} results available for this problem. The results developed in this report can be viewed as the first attempt to addressing this problem for nonlinear systems with autonomous self-regulating time lags.

Note that the limiting delay system (\ref{limiting}) has the same equilibrium solutions as (\ref{DS}). Based on this new notion, we have the following convergence result. 

\begin{lemma}\label{lemma3}
Consider the nonlinear system (\ref{DS}). Let $\mathcal{D}\subseteq\mathcal{C}$ be invariant with respect to (\ref{DS}). Assume that the solutions of (\ref{DS}) are bounded. Furthermore, assume (\ref{limiting}) is a limiting delay system of (\ref{DS}) relative to $\mathcal{D}$. Then $\omega(\phi)$ is invariant with respect to (\ref{limiting}) for every initial condition $x_{0}(\phi)=\phi\in\mathcal{D}$.
\end{lemma}

\begin{proof}
Since the solutions of (\ref{DS}) are bounded, it follows from Lemma 1.3 of \cite[p.~103]{HV:1993} that the set $\omega(\phi)$ is nonempty. Let $z\in\omega(\phi)$ and hence, there exists a nonnegative sequence $\{t_{n}\}_{n=1}^{\infty}$ such that $t_{n}\to\infty$ and $x(\phi)(t_{n})\to z$ as $n\to\infty$. For each $n=1,2,\ldots$ and the solution $x(\cdot)$ to (\ref{DS}) with $x(s)=\phi(s)$, $s\in[-h,0]$, define the continuous functions $x_{n}(\tau):=x(\tau+t_{n})$ and $f_{n}(t):=|\mathcal{X}_{n}(t+t_{n})|+\frac{1}{n}$, $\tau\geq-h$, $t\geq 0$,
where
$\mathcal{X}_{n}(t+t_{n}):=g(x_{n}(t-\tau_{1}(t+t_{n})),\ldots,x_{n}(t-\tau_{m}(t+t_{n})))-g(x_{n}(t-h_{1}),\ldots,x_{n}(t-h_{m}))$.

Let $g(\mathfrak{D}x(t)):=g(x(t-h_{1}),\ldots,x(t-h_{m}))$, where $\mathfrak{D}$ denotes a multiple delay operator. Clearly the map $t\mapsto f_{n}(t)$ is continuous for every $n$.
By Definition~\ref{def}, note that for every $\varepsilon>0$, there exists an integer $N\geq 1$ such that for every $n>N$,
$0<f_{n}(t)<\varepsilon$ for all $t\geq 0$. Observe that $x_{n}(0)\to z(0)$ as $n\to\infty$.
Then it follows that for all $n > N$, we have
\begin{eqnarray}\label{eq:diffincl}
\dot{x}_{n}(t)\in\{f(x_{n}(t))\}+\overline{\mathcal{B}}_{f_{n}(t)}(g(\mathfrak{D}x_{n}(t)))\subset\{f(x_{n}(t))\}+\overline{\mathcal{B}}_{\varepsilon}(g(\mathfrak{D}x_{n}(t))),\,\,{\rm{a.e.}}\,\, t\geq 0,
\end{eqnarray} where $\overline{\mathcal{B}}_r(s)$ denotes the closed ball centered at $s$ with radius $r$ and ``a.e." denotes almost everywhere in the sense of Lebesgue measure. Note that (\ref{eq:diffincl}) represents a set-valued differential equation called \textit{differential inclusion} in the literature \cite{Clarke:1983,AC:1984}. In fact, (\ref{eq:diffincl}) is a \textit{delay differential inclusion} \cite{HP:PAMS:1995}.

Let $I_{j}$ denote the interval $[j,j+1]$, $j=0,1,2,\ldots$. It follows from (\ref{eq:diffincl}) and boundedness of $x(\cdot)$ that $|\dot{x}_{n}(t)|\leq M$ for almost all $t\geq0$ and some $M>0$. Hence, the functions $x_{n}(t)$ are uniformly bounded and equicontinuous. By the Arzel\`{a}-Ascoli theorem (see, e.g., \cite{Carothers:2000}) and Theorem 3.1.7 of \cite{Clarke:1983} or Lemma 4.5 of \cite{VP:JMAA:1982}, the sequence of restricted functions $\{x_{n}(t)\}$, $t\in I_{0}$, has a subsequence $\{x_{\sigma_{1}(n)}(t)\}$, $t\in I_{0}$, converging uniformly as $n \rightarrow \infty$ to an absolutely continuous function $x_{1}^{*}:I_{0}\to\mathbb{R}^{n}$,  satisfying (\ref{limiting}) almost everywhere and with the endpoint $x_{1}^{*}(0)=z(0)$. By repeating the same argument, the sequence $\{x_{\sigma_{1}(n)}(t)\}$, $t\in I_{1}$, contains a subsequence $\{x_{\sigma_{2}(n)}(t)\}$, $t\in I_{1}$, converging uniformly as $n \rightarrow \infty$ to an absolutely continuous function $x_{2}^{*}:I_{1}\to\mathbb{R}^{n}$, satisfying (\ref{limiting}) almost everywhere and with $x_{2}^{*}(1)=x_{1}^{*}(1)$. By induction, there exists a nested sequence $\{x_{\sigma_{\ell}(n)}\}_{\ell=1}^{\infty}$ of subsequences of $\{x_{n}\}$ such that for each $\ell=1,2,\ldots$, $\{x_{\sigma_{\ell}(n)}\}$ converges uniformly on $I_{\ell-1}=[\ell-1,\ell]$ to an absolutely continuous function $x_{\ell}^{*}: I_{\ell-1}\to\mathbb{R}^{n}$ that satisfies (\ref{limiting}) almost everywhere and with $x_{\ell+1}^{*}(\ell)=x_{\ell}^{*}(\ell)$ and $x_{1}^{*}(0)=z(0)$.

Define $x^{*}:[0,\infty)\to\mathbb{R}^{n}$ as
$x^{*}(t)=x_{\ell}^{*}(t)$ for $t\in I_{\ell-1}$ and $x^{*}(\theta)=z(\theta)$, $\theta\in[-h,0]$.
Then $x^{*}(\cdot)$ satisfies (\ref{limiting}) and $x^{*}(\theta)=z(\theta)$, $\theta\in[-h,0]$. Using the Cantor diagonal argument (see, e.g., \cite[p.210]{Carothers:2000}), it follows that there exists a subsequence $\{x_{\sigma_{n}(n)}\}$ of $\{x_{n}\}$ converging uniformly on every interval $[-h,\ell]$ to $x^{*}$ for all $\ell$. Since the sequence $\{t_{\sigma_{n}(n)}\}$ is such that $t_{\sigma_{n}(n)}\to\infty$ and $x(t+t_{\sigma_{n}(n)})=x_{\sigma_{n}(n)}(t)\to x^{*}(t)$ as $n\to\infty$, it follows that $x^{*}(t)\in\omega(\phi)$ for all $t\in[-h,\infty)$, which implies that $\omega(\phi)$ is invariant with respect to (\ref{limiting}).
\end{proof}

\begin{lemma}\label{convergent_eq}
Consider (\ref{limiting}). Let $\mathcal{D}\subseteq\mathcal{C}$ be invariant with respect to (\ref{limiting}). If the solutions of (\ref{limiting}) converge, that is, $\lim_{t\to\infty}z_{t}(\phi)$ exists for every $\phi\in\mathcal{D}$, then the function $\Omega:\mathcal{D}\to\mathcal{D}$ defined by $\Omega(\phi)=\lim_{t\to\infty}z_{t}(\phi)$, $\phi\in\mathcal{D}$, satisfies that for each $\phi\in\mathcal{D}$, $\Omega(\phi)$ is in the equilibrium set for (\ref{limiting}).
\end{lemma}

\begin{proof}
It follows from continuity of the solutions to (\ref{limiting}) that for every $s\geq 0$, $z_{s}(\Omega(\phi))=\lim_{t\to\infty}z_{t+s}(\phi)=\Omega(\phi)$. Thus, for every $\phi\in\mathcal{D}$, $\Omega(\phi)$ is in the equilibrium set for (\ref{limiting}).
\end{proof}

Now we have the main result for this report.

\begin{theorem}\label{thm1}
Consider the nonlinear system (\ref{DS}). Let $\mathcal{D}\subseteq\mathcal{C}$ be invariant with respect to (\ref{DS}). Assume (\ref{DS}) is Lyapunov stable relative to $\mathcal{D}$. Furthermore, assume (\ref{limiting}) is a limiting delay system of (\ref{DS}) relative to $\mathcal{D}$ and (\ref{limiting}) is semistable relative to $\mathcal{D}$. Then (\ref{DS}) is semistable relative to $\mathcal{D}$.
\end{theorem}

\begin{proof}
Since by assumption, (\ref{DS}) is Lyapunov stable relative to $\mathcal{D}$, it follows that the solutions of (\ref{DS}) are bounded. Then by Lemma~\ref{lemma3}, $\omega(\phi)$ is invariant with respect to (\ref{limiting}). Next, since by semistability, the solutions of (\ref{limiting}) converge, it follows that $\omega(\phi)$ contains positive limit points of (\ref{limiting}), and hence, $\omega(\phi)$ contains the positive limit set of (\ref{limiting}). Furthermore, it follows from Lemma~\ref{convergent_eq} that the positive limit set of (\ref{limiting}) contains an equilibrium solution of (\ref{limiting}). This equilibrium solution is also an equilibrium solution of (\ref{DS}) and by assumption, it is Lyapunov stable relative to $\mathcal{D}$. Hence, $\omega(\phi)$ contains a Lyapunov stable equilibrium solution for (\ref{DS}). Now it follows from Lemma~\ref{omega} that the solution of (\ref{DS}) converges to this Lyapunov stable equilibrium solution, which implies convergence of the solution of (\ref{DS}). By definition, (\ref{DS}) is semistable relative to $\mathcal{D}$.
\end{proof}

\begin{remark}
To discuss semistability of (\ref{DS}) using Theorem~\ref{thm1}, one has to know the information on Lyapunov stability of (\ref{DS}). Note that here we only assume $\tau_{k}(t)$ is \textit{continuous} for every $k=1,\ldots,m$. Hence, it is very difficult to use the Lyapunov-Krasovskii functional approach \cite{Kra:1963,HV:1993} to prove the Lyapunov stability of (\ref{DS}) since it requires the first-order derivative of $\tau_{k}(t)$. In this case, the Lyapunov stability of (\ref{DS}) may be verified using Razumikhin theorems via Lyapunov-Razumikhin functions \cite{Raz:PMM:1956,Raz:AT:1960,HV:1993}.
\end{remark}

\begin{example}\label{eg_scalar1}
Consider the scalar time-delay system given by
\begin{eqnarray}\label{scalar}
\dot{x}(t)=-x(t)+x(t-\tau(t)),
\end{eqnarray} where $x(t)\in\mathbb{R}$, $\tau(\cdot)$ is continuous, and $0\leq\tau(t)\leq h$ for all $t\in\mathbb{R}$. Consider the Lyapunov-Razumikhin function given by $V(x)=(x-\alpha)^{2}/2$, where $\alpha$ is an arbitrary constant. Then it follows from Theorem 4.1 of Chapter 5 of \cite{HV:1993} that (\ref{scalar}) is Lyapunov stable relative to $\mathcal{C}$. See \cite[p.~154]{HV:1993} for a detailed proof. 
\end{example}

\begin{remark}\label{SWLS}
Suppose the solutions of (\ref{DS}) are bounded. If $|f(x)|\leq\beta(|x|)$, $\beta(\cdot)$ is a class $\mathcal{K}$ function, and $\lim_{t\to\infty}\tau_{k}(t)=h_{k}$ for every $k=1,\ldots,m$, then (\ref{limiting}) is a limiting delay system of (\ref{DS}) relative to $\mathcal{C}$. Note that the assumption $\lim_{t\to\infty}\tau_{k}(t)=h_{k}$ implies that there exists $h>0$ such that $0\leq\tau_{k}(t)\leq h$ for every $t\geq0$ and every $k=1,\ldots,m$. Furthermore, it implies that the time-delays of (\ref{DS}) are autonomous self-regulating time lags \cite{BC:PNAS:1959,Cooke:JDE:1970,Rek:DAN:1958,Ginzburg:DU:1970,Ruiz:JDE:1976,Stephan:JDE:1969,Stephan:JDE:1970}. To show that (\ref{limiting}) is a limiting delay system of (\ref{DS}) relative to $\mathcal{C}$, let $|x(t)|\leq M$, $t\in[-h,\infty)$. Then from (\ref{DS}),
$|\dot{x}(t)|\leq\beta(M)+mLM+|g(0,\ldots,0)|:=K$, where $L$ is the Lipschitz constant for $g$.
Because
$x(t-\tau_{k}(t))-x(t-h_{k})=\int_{t-h_{k}}^{t-\tau_{k}(t)}\dot{x}(s)ds$,
it follows that
$|x(t-\tau_{k}(t))-x(t-h_{k})|\leq K|\tau_{k}(t)-h_{k}|$. Hence, $|\mathcal{X}(t)|\leq L\sum_{k=1}^{m}|x(t-\tau_{k}(t))-x(t-h_{k})|\leq KL\sum_{k=1}^{m}|\tau_{k}(t)-h_{k}|$.
Thus, if $\lim_{t\to\infty}\tau_{k}(t)=h_{k}$, then for any divergent sequence $\{t_{n}\}_{n=1}^{\infty}$ and $t\geq0$,
$\lim_{n\to\infty} \mathcal{X}(t_{n}+t)=0$. By definition, (\ref{limiting:sys1}) is a limiting delay system of (\ref{eq:sys1}) relative to $\mathcal{C}$.
\end{remark}

\begin{example}
Consider the time-delay system given by (\ref{scalar}) where $\tau(t)=h|\sin(\pi/2+\pi/(1+|t|))|$, $t\in\mathbb{R}$. Clearly $\tau(\cdot)$ is continuous but not differentiable for all $t\in\mathbb{R}$. We claim that
\begin{eqnarray}\label{hr}
\dot{z}(t)=-z(t)+z(t-h)
\end{eqnarray} is a limiting delay system of (\ref{scalar}) relative to $\mathcal{C}$. To see this, note that $\lim_{t\to\infty}\tau(t)=h$. Now it follows from Remark~\ref{SWLS} that (\ref{hr}) is a limiting delay system of (\ref{scalar}) relative to $\mathcal{C}$. 
\end{example}

Next, motivated by \cite{HHB:TAC:2009}, we present a Lyapunov-type result for semistability of nonlinear systems with constant time delays using Lyapunov-Krasovskii-type functionals. This result will help us determine the semistability of (\ref{limiting}) which is required by Theorem \ref{thm1}.

\begin{theorem}\label{LKF}
Consider the dynamical system (\ref{limiting}). Assume the solutions of (\ref{limiting}) are bounded and there exists a continuous functional $V:\mathcal{C}\to\mathbb{R}$ such that $\dot{V}$ is defined on $\mathcal{C}$ and $\dot{V}(\phi)\leq 0$ for all $\phi\in\mathcal{C}$. If every point in the largest invariant set $\mathcal{M}$ of $\dot{V}^{-1}(0):=\{x\in\mathcal{C}:\dot{V}(x)=0\}$ is a Lyapunov stable equilibrium solution of (\ref{limiting}) relative to $\mathcal{C}$, then (\ref{limiting}) is semistable relative to $\mathcal{C}$.
\end{theorem}

\begin{proof}
Since every solution is bounded, it follows from
the hypotheses on $V$ that, for every $\phi\in\mathcal{C}$, the positive limit set of (\ref{limiting}) denoted by $\varpi(\phi)$ is nonempty
and contained in the largest invariant subset $\mathcal{M}$ of $\dot{V}^{-1}(0)$.
Since every point in $\mathcal{M}$ is a Lyapunov stable equilibrium solution
of (\ref{limiting}), it follows from Lemma~\ref{omega} that $\varpi(\phi)$
contains a single point for every $\phi\in\mathcal{C}$ and the solutions of (\ref{limiting}) converge.
Since $\Omega(\phi)$ is Lyapunov stable for every $\phi\in\mathcal{C}$,
semistability follows.
\end{proof}

\begin{example}\label{LKF_eg1}
Consider the time-delay system given by (\ref{hr}). Let $V(z_{t})=z^{2}(t)+\int_{t-h}^{t}z^{2}(s){\rm{d}}s$. Then the derivative of $V(\cdot)$ along the solutions of (\ref{hr}) is given by $\dot{V}(z_{t})=2z(t)\dot{z}(t)+z^{2}(t)-z^{2}(t-h)=-(-z(t)+z(t-h))^{2}\leq 0$, $t\geq 0$. Note that $\dot{V}^{-1}(0)=\{\phi(\cdot)\in\mathcal{C}:-\phi(0)+\phi(-h)=0\}$. Furthermore, the largest invariant set contained in $\dot{V}^{-1}(0)$ is given by $\mathcal{M}=\{z_{t}\in\mathcal{C}:z(t)=\alpha\in\mathbb{R},t\in[-h,\infty)\}$. Now, it follows from Example~\ref{eg_scalar1} and Theorem~\ref{LKF} that (\ref{hr}) is semistable relative to $\mathcal{C}$, and hence, by Example~\ref{eg_scalar1} and Theorem~\ref{thm1}, (\ref{scalar}) with $\tau(t)=h|\sin(\pi/2+\pi/(1+|t|))|$ is semistable relative to $\mathcal{C}$. 
\end{example}

As an alternative to Theorem \ref{LKF}, we present a Lyapunov-Razumikhin function approach to semistability analysis of nonlinear systems with constant time delays. Motivated by \cite{HT:JDE:1983}, this result gives a different method to prove semistability of (\ref{limiting}) other than Theorem~\ref{LKF}, which is useful for many cases in that constructing a Lyapunov-Krasovskii functional for (\ref{limiting}) may not be an easy task in these cases.

\begin{theorem}\label{IP_Razu}
Consider the dynamical system (\ref{limiting}). Assume the solutions of (\ref{limiting}) are bounded and there exists a continuous function $V:\mathbb{R}^{n}\to\mathbb{R}$ such that $\dot{V}$ is defined on $\mathbb{R}^{n}$ and $\dot{V}(\phi(0))\leq 0$ for all $\phi\in\mathcal{C}$ such that $V(\phi(0))=\max_{-h\leq s\leq 0}V(\phi(s))$. If every point in the largest invariant set $\mathcal{M}$ of $\mathcal{R}:=\{\phi\in\mathcal{C}:\max_{s\in[-h,0]}V(z_{t}(\phi)(s))=\max_{s\in[-h,0]}V(\phi(s)),\forall t\geq 0\}$ is a Lyapunov stable equilibrium solution of (\ref{limiting}) relative to $\mathcal{C}$, then (\ref{limiting}) is semistable relative to $\mathcal{C}$.
\end{theorem}

\begin{proof}
Let $\phi\in\mathcal{C}$ be such that $z_{t}(\phi)$ is bounded on $[-h,\infty)$. Then $\varpi(\phi)$ is nonempty. Using a standard Razumikhin-type argument (see the proof of Theorem 4.1 in \cite[p.~152]{HV:1993}) and the assumptions on $V$, it follows that the function $\max_{-h\leq s\leq 0}V(z_{t}(\phi)(s))$ is a nonincreasing function of $t$ on $[0,\infty)$. Since $V$ is bounded from below along this solution, $\lim_{t\to\infty}\{\max_{-h\leq s\leq 0}V(z_{t}(\phi)(s))\}$ exists. Hence, $\varpi(\phi)\subseteq\mathcal{M}\subseteq\mathcal{R}$ and $z_{t}(\phi)\to\mathcal{M}$ as $t\to\infty$. Finally, since every point in $\mathcal{M}$ is a Lyapunov stable equilibrium solution of (\ref{limiting}) relative to $\mathcal{C}$, it follows from Lemma~\ref{omega} that the solutions of (\ref{limiting}) converge. Thus, by definition, (\ref{limiting}) is semistable relative to $\mathcal{C}$.
\end{proof}

\begin{example}
Consider the scalar time-delay system given by
\begin{eqnarray}\label{neutral}
\dot{x}(t)=-mx(t)+\sum_{k=1}^{m}x(t-\tau_{k}(t)),
\end{eqnarray} where $x(t)\in\mathbb{R}$ and $\tau_{k}(t)=(hk/m)|\sin(\pi/2+\pi/(1+|t|))|$ for every $k=1,\ldots,m$, $t\in\mathbb{R}$. Using the Lyapunov-Razumikhin function $V(x-\alpha)=(x-\alpha)^{2}/2$ and similar arguments as in \cite[p.~154]{HV:1993}, it follows that (\ref{neutral}) is Lyapunov stable relative to $\mathcal{C}$. Next, note that
\begin{eqnarray}\label{neutral_limiting}
\dot{z}(t)=-mz(t)+\sum_{k=1}^{m}z(t-hk/m)
\end{eqnarray} is a limiting delay system of (\ref{neutral}) relative to $\mathcal{C}$. We show that (\ref{neutral_limiting}) is semistable relative to $\mathcal{C}$. To see this, note that for $V(z-\alpha)=(z-\alpha)^{2}/2$, $\alpha\in\mathbb{R}$, we have
$\dot{V}(z(t)-\alpha)=-m[z(t)-\alpha]^{2}+\sum_{k=1}^{m}[z(t)-\alpha][z(t-hk/m)-\alpha]\leq-m[z(t)-\alpha]^{2}+\sum_{k=1}^{m}|z(t)-\alpha||z(t-hk/m)-\alpha|\leq-m[z(t)-\alpha]^{2}+\sum_{k=1}^{m}[z(t)-\alpha]^{2}=0$
if $|z(t+\theta)-\alpha|\leq|z(t)-\alpha|$ for $\theta\in[-h,0]$. Hence, it follows from Theorem 4.1 of \cite[p.~152]{HV:1993} that (\ref{neutral_limiting}) is Lyapunov stable relative to $\mathcal{C}$. In particular, note that $V(z)=z^{2}/2$ and $\dot{V}(z(t))=-mz^{2}(t)+\sum_{k=1}^{m}z(t)z(t-hk/m)\leq0$ if $|z(t+\theta)|\leq|z(t)|$ for $\theta\in[-h,0]$. Next, we want to compute $\mathcal{R}$ and $\mathcal{M}$. First of all, note that $\mathcal{M}$ is nonempty since $0\in\mathcal{M}$. Let $\phi\in\mathcal{R}$, that is, let $\phi\in\mathcal{C}$ be such that
$\max_{-h\leq\theta\leq0}|z_{t}(\phi)(\theta)|=\max_{-h\leq\theta\leq0}|\phi(\theta)|$ for all $t\geq 0$. For $\phi\in\mathcal{R}$ satisfying $|\phi(0)|\geq|\phi(s)|$, $s\in[-h,0]$, there exists $t^{*}>0$ for which $\dot{V}(z_{t^{*}}(\phi))=0$ as $V$ attains a relative maximum for such $t^{*}$. For such a $t^{*}$, $-mz^{2}(t^{*})+\sum_{k=1}^{m}z(t^{*})z(t^{*}-hk/m)=0$, and hence, $z(t^{*})=0$ or $-mz(t^{*})+\sum_{k=1}^{m}z(t^{*}-hk/m)=0$. Consider the case where $\phi(0)\geq0$. Then it follows that $z(t^{*}+\theta)\leq z(t^{*})$ for all $\theta\in[-h,0]$. Therefore, $z(t^{*})=0$ or $z(t^{*})=z(t^{*}-hk/m)$ for every $k=1,\ldots,m$. Due to uniqueness of solutions, we have $z_{t}(\phi)=0$ or $z_{t}(\phi)=z_{t}(\phi)(-kh/m)$ for every $k=1,\ldots,m$ and all $t\geq t^{*}$. By Example~\ref{LKF_eg1}, it follows that $\mathcal{M}=\{z_{t}\in\mathcal{C}:z(t)=\alpha\in\mathbb{R},t\in[-h,\infty)\}$. Now, it follows from Theorem~\ref{IP_Razu} that (\ref{neutral_limiting}) is semistable relative to $\mathcal{C}$. Finally, it follows from Theorem~\ref{thm1} that (\ref{neutral}) is semistable relative to $\mathcal{C}$. 
\end{example}

\subsection{Specialization to the Consensus Problem with Autonomous Self-Regulating Time Lags}

In this subsection, we discuss the delay effect on the consensus problem (\ref{eq:sys1}) with autonomous self-regulating time lags. In particular, we show that under certain mild conditions, (\ref{eq:sys1}) is a semistable system, implying that it has no periodic solutions in $\mathcal{C}$. This contribution advances the study of delay effects on the solution of differential equations with autonomous self-regulating time lags and complements the relevant results in the literature \cite{BC:PNAS:1959,Cooke:JDE:1970,Rek:DAN:1958,Ginzburg:DU:1970,Ruiz:JDE:1976,Stephan:JDE:1969,Stephan:JDE:1970}.

\begin{lemma}\label{convergence}
Consider the dynamical system (\ref{eq:sys1}). Assume the solutions of (\ref{eq:sys1}) are bounded. If $\lim_{t\to\infty}\tau_{k}(t)=h_{k}$ for every $k=1,\ldots,m$, then
\begin{eqnarray}\label{limiting:sys1}
\dot{z}(t)=Ez(t)+\sum_{k=1}^{m}F_{k}z(t-h_{k})
\end{eqnarray}
is a limiting delay system of (\ref{eq:sys1}) relative to $\mathcal{C}$.
\end{lemma}
\begin{proof}
The proof is essentially given by Remark \ref{SWLS}, and hence, is omitted here.
\end{proof}

Next, we present a Lyapunov stability result for  (\ref{eq:sys1}). Define $F:=\sum_{k=1}^{m}F_{k}$. For a matrix $A\in\mathbb{R}^{m\times n}$, we use $A_{(i,j)}$ to denote the $(i,j)$th element of $A$.

\begin{lemma}\label{LSTV}
Consider the dynamical system (\ref{eq:sys1}) having the following structure: given $a_{ij}\geq 0$, $i,j=1,\ldots,n$, all the elements in $F_{k}$ are nonnegative, $k=1,\ldots,m$, $F:=\sum_{k=1}^{m}F_{k}$, 
\begin{eqnarray}
E_{(i,j)}=\left\{\begin{array}{ll}
-\sum_{k=1}^{n}a_{ki}, & i=j,\\
0, & i\neq j,
\end{array}\right.,\quad F_{(i,j)}=\left\{\begin{array}{ll}
0, & i=j,\\
a_{ij}, & i\neq j,
\end{array}\right.\label{Fij}
\end{eqnarray} $i,j=1,\ldots,n$. Then (\ref{eq:sys1}) is Lyapunov stable relative to $\mathcal{C}$.
\end{lemma}

\begin{proof}
Note that under the assumptions in Lemma~\ref{LSTV}, (\ref{eq:sys1}) can be rewritten as
\begin{eqnarray}\label{component}
\dot{z}_{i}(t)&=&-\sum_{k=1}^{m}\sum_{j=1}^{n}F_{k(i,j)}z_{i}(t)+\sum_{k=1}^{m}\sum_{j=1}^{n}F_{k(i,j)}z_{j}(t-\tau_{k}(t))\nonumber\\
&=&\sum_{k=1}^{m}\sum_{j=1}^{n}F_{k(i,j)}(z_{j}(t-\tau_{k}(t))-z_{i}(t)),
\end{eqnarray} where $F_{k(i,j)}\geq 0$ for all $k=1,\ldots,m$ and $i,j=1,\ldots,n$.

Consider the Lyapunov-Razumikhin function given by
$V(x)=\frac{1}{2}\max_{1\leq i\leq n}\{(x_{i}-\alpha)^{2}\}$,
where $\alpha\in\mathbb{R}$. To use Razumikhin theorems (see \cite[p.~151]{HV:1993}) showing Lyapunov stability, we focus on
$V(\phi(0))=\max_{-h\leq s\leq 0}V(\phi(s))$, that is, for the cases in which
\begin{eqnarray}\label{phiI}
(\phi_{I}(0)-\alpha)^{2}\geq(\phi_{j}(s)-\alpha)^{2}, \quad s\in[-h,0],
\end{eqnarray}
where $I$ is the index for which $|\phi_{I}(0)-\alpha|=\max_{1\leq i\leq n}|\phi_{i}(0)-\alpha|$.

The derivative of $V$ along the solutions of (\ref{eq:sys1}) is given by $\dot{V}(\phi(0))=(\phi_{I}(0)-\alpha)\dot{\phi}_{I}(0)$. First suppose $\phi_{I}(0)-\alpha=c>0$. Then it follows from (\ref{phiI}) that $-c\leq\phi_{j}(s)-\alpha\leq c$ for all $j=1,\ldots,n$ and $s\in[-h,0]$. This implies that $\phi_{j}(s)-\phi_{I}(0)\leq 0$ for all $j=1,\ldots,n$ and $s\in[-h,0]$. Hence, by (\ref{component}), it follows that $\dot{\phi}_{I}(0)\leq 0$. Similarly, one can show that if $\phi_{I}(0)-\alpha\leq 0$, then $\dot{\phi}_{I}(0)\geq 0$. In summary, $\dot{V}(\phi(0))\leq 0$. Now it follows from Theorem 4.1 of \cite[p.~152]{HV:1993} that (\ref{eq:sys1}) is Lyapunov stable relative to $\mathcal{C}$.
\end{proof}

The following corollary regarding semistability of time-varying delay network consensus protocols given by (\ref{eq:sys1}) follows directly from Lemmas \ref{convergence} and \ref{LSTV}, and Theorems \ref{thm1} and \ref{coro1b}. To state this result, define $\textbf{1}=[1,\ldots,1]^{\rm{T}}\in\mathbb{R}^{n}$.

\begin{corollary}\label{coro1a}
Consider the dynamical system (\ref{eq:sys1}) having the structure given by (\ref{Fij}).
Assume $\lim_{t\to\infty}\tau_{k}(t)=h_{k}$ for every $k=1,\ldots,m$. Furthermore, assume that $(E+F)^{\rm{T}}\textbf{1}=(E+F)\textbf{1}=0$ and rank$(E+F)=n-1$. Then for every $\alpha\in\mathbb{R}$, $\alpha\textbf{1}$ is a semistable equilibrium solution of (\ref{eq:sys1}) relative to $\mathcal{C}$. Furthermore, $x(t)\to\alpha^{*}\textbf{1}$ as $t\to\infty$, where
\begin{eqnarray}\label{lea}
\alpha^{*}=\frac{\textbf{1}^{\rm{T}}\phi(0)+\sum_{k=1}^{m}\int_{-h_{k}}^{0}\textbf{1}^{\rm{T}}F_{k}\phi(\theta){\rm{d}}\theta}{n+\sum_{k=1}^{m}h_{k}\textbf{1}^{\rm{T}}F_{k}\textbf{1}}.
\end{eqnarray}
\end{corollary}

\begin{proof}
First, it follows from Lemma~\ref{convergence} that
\begin{eqnarray}\label{lshk}
\dot{x}(t)=Ex(t)+\sum_{k=1}^{m}F_{k}x(t-h_{k})
\end{eqnarray} is a limiting delay system of (\ref{eq:sys1}) relative to $\mathcal{C}$. Next, it follows from Theorem 3.1 of \cite{CHHR:SCL:2008} that (\ref{lshk}) is semistable relative to $\mathcal{C}$. Now, it follows from Lemma~\ref{LSTV} and Theorem~\ref{thm1} that (\ref{eq:sys1}) is semistable relative to $\mathcal{C}$. The expression (\ref{lea}) for $\alpha^{*}$ follows from (\ref{nea}).
\end{proof}

\begin{example}
Consider the time-delay system given by
\begin{eqnarray}
\dot{x}_{1}(t)&=&-x_{1}(t)+x_{2}(t-\tau_{1}(t)),\label{2nd_1}\\
\dot{x}_{2}(t)&=&-x_{2}(t)+x_{1}(t-\tau_{2}(t)),\label{2nd_2}
\end{eqnarray} where $\tau_{1}(t)=h_{1}|t\sin(1/t)|$ for $t\neq 0$ and $\tau_{1}(t)=h_{1}$ for $t=0$, and $\tau_{2}(t)=h_{2}(1-e^{-|t|})$. Clearly in this case,
$E=\Big[\begin{array}{cc}
-1 & 0 \\
0 & -1 \\
\end{array}\Big]$, $F_{1}=\Big[\begin{array}{cc}
0 & 1 \\
0 & 0 \\
\end{array}\Big]$, and $F_{2}=\Big[\begin{array}{cc}
0 & 0 \\
1 & 0 \\
\end{array}\Big]$. 
Now it follows from Corollary~\ref{coro1a} that the time-delay system given by (\ref{2nd_1}) and (\ref{2nd_2}) is semistable relative to $\mathcal{C}$. 
\end{example}

Next, we generalize Corollary \ref{coro1a} to the nonlinear system given by
\begin{eqnarray}\label{NNC}
\dot{x}(t)=f(x(t))+\sum_{k=1}^{m}g_{k}(x(t-\tau_{k}(t))),
\end{eqnarray} where $f(x)=[f_{1}(x_{1}),\ldots,f_{q}(x_{n})]^{\rm{T}}$. In particular, we present the following stability result for the nonlinear network consensus with time-varying delays given by the form of (\ref{NNC}). Recall that for a diagonal matrix $\Lambda\in\mathbb{R}^{n\times n}$, the \textit{Drazin inverse} $\Lambda^{\rm{D}}\in\mathbb{R}^{n\times n}$ is given by $\Lambda_{(i,i)}^{\rm{D}}=0$ if $\Lambda_{(i,i)}=0$ and $\Lambda_{(i,i)}^{\rm{D}}=1/\Lambda_{(i,i)}$ if $\Lambda_{(i,i)}\neq 0$, $i=1,\ldots,n$ \cite[p.~401]{Bernstein:2009}.

\begin{theorem}\label{coro1b}
Consider the dynamical system (\ref{NNC}) where $f(0)=0$, $g_{k}(0)=0$, $k=1,\ldots,m$, and $f_{i}(\cdot)$ is strictly decreasing for $f_{i}\not\equiv0$, $i=1,\ldots,n$. Assume (\ref{NNC}) is Lyapunov stable relative to $\mathcal{C}$ and $\lim_{t\to\infty}\tau_{k}(t)=h_{k}$ for every $k=1,\ldots,m$. Next, assume that $\textbf{1}^{\rm{T}}(f(x)+\sum_{k=1}^{m}g_{k}(x))=0$ for all $x\in\mathbb{R}^{n}$ and $f(x)+\sum_{k=1}^{m}g_{k}(x)=0$ if and only if $x=\alpha\textbf{1}$ for any $\alpha\in\mathbb{R}$. Furthermore, assume that there exist nonnegative diagonal matrices $P_{k}\in\mathbb{R}^{n\times n}$, $k=1,\ldots,m$, such that $P:=\sum_{k=1}^{m}P_{k}>0$, $P_{k}^{\rm{D}}P_{k}g_{k}(x)=g_{k}(x)$ for every $x\in\mathbb{R}^{n}$ and $k=1,\ldots,m$, 
$\sum_{k=1}^{m}[g_{k}(x)-g_{k}(\alpha\textbf{1})]^{\rm{T}}P_{k}[g_{k}(x)-g_{k}(\alpha\textbf{1})]\leq[f(x)-f(\alpha\textbf{1})]^{\rm{T}}P[f(x)-f(\alpha\textbf{1})]$ for any $x\in\mathbb{R}^{n}$ and $\alpha\in\mathbb{R}$, and 
$\sum_{k=1}^{m}[f(x)-f(\alpha\textbf{1})]^{\rm{T}}PP_{k}^{\rm{D}}P[f(x)-f(\alpha\textbf{1})]\leq [f(x)-f(\alpha\textbf{1})]^{\rm{T}}P[f(x)-f(\alpha\textbf{1})]$ for any $x\in\mathbb{R}^{n}$ and $\alpha\in\mathbb{R}$.
Then for every $\alpha\in\mathbb{R}$, $\alpha\textbf{1}$ is a semistable equilibrium solution of (\ref{NNC}) relative to $\mathcal{C}$. Furthermore, $x(t)\to\alpha^{*}\textbf{1}$ as $t\to\infty$, where $\alpha^{*}$ satisfies
\begin{eqnarray}\label{nea}
n\alpha^{*}+\sum_{k=1}^{m}h_{k}\textbf{1}^{\rm{T}}g_{k}(\alpha^{*}\textbf{1})=\textbf{1}^{\rm{T}}\phi(0)+\sum_{k=1}^{m}\int_{-h_{k}}^{0}\textbf{1}^{\rm{T}}g_{k}(\phi(\theta)){\rm{d}}\theta.
\end{eqnarray}
\end{theorem}

\begin{proof}
First, it follows from Lemma~\ref{convergence} that
\begin{eqnarray}\label{nshk}
\dot{x}(t)=f(x(t))+\sum_{k=1}^{m}g_{k}(x(t-h_{k}))
\end{eqnarray} is a limiting delay system of (\ref{NNC}) relative to $\mathcal{C}$. Next, it follows from the similar arguments as in the proof of Theorem 4.1 of \cite{CHHR:SCL:2008} that (\ref{nshk}) is semistable relative to $\mathcal{C}$. 
Now, it follows from Theorem~\ref{thm1} that (\ref{NNC}) is semistable relative to $\mathcal{C}$. The equation (\ref{nea}) follows from Theorem 4.1 of \cite{CHHR:SCL:2008}.
\end{proof}

\begin{remark}
Theorem~\ref{coro1b} requires Lyapunov stability of (\ref{NNC}) \textit{a priori} in order to test semistability of (\ref{NNC}). In general this requirement cannot be fulfilled by the conditions in Theorem~\ref{coro1b}. However, for some special cases, one can use Razumikhin theorems to prove Lyapunov stability of (\ref{NNC}). For instance, consider (\ref{NNC}) in which $f_{i}(x)=-(n-1)\sigma(x_{i})$, $i=1,\ldots,n$, $g_{k}(x)=\textbf{e}_{k}\sum_{i=1}^{n}\sigma(x_{i})$, $k=1,\ldots,n$, $\sigma(0)=0$, and $\sigma(\cdot)$ is strictly increasing, where $\textbf{e}_{k}\in\mathbb{R}^{n}$ denotes the elementary
vector of order $n$ with 1 in the $i$th component and 0's elsewhere. In this case, (\ref{NNC}) is Lyapunov stable relative to $\mathcal{C}$. The idea of the proof is illustrated by the following example. 
\end{remark}

\begin{example}
Consider the time-delay system given by
\begin{eqnarray}
\dot{x}_{1}(t)&=&-\sigma(x_{1}(t))+\sigma(x_{2}(t-\tau_{1}(t))),\label{2nd_n1}\\
\dot{x}_{2}(t)&=&-\sigma(x_{2}(t))+\sigma(x_{1}(t-\tau_{2}(t))),\label{2nd_n2}
\end{eqnarray} where $\sigma(x)=x+\tanh x$, $\tau_{1}(t)=h_{1}-h_{1}e^{-|t|}\sin t$, and $\tau_{2}(t)=h_{2}-h_{2}\sin(1/(1+|t|))$. In this case,
$f(x)=\Big[\begin{array}{c}
-\sigma(x_{1})\\
-\sigma(x_{2})\\
\end{array}\Big]$, $g_{1}(x)=\Big[\begin{array}{c}
\sigma(x_{2})\\
0\\
\end{array}\Big]$, $g_{2}(x)=\Big[\begin{array}{c}
0\\
\sigma(x_{1})\\
\end{array}\Big]$, $P_{1}=\Big[\begin{array}{cc}
1 & 0 \\
0 & 0 \\
\end{array}\Big]$, $P_{2}=\Big[\begin{array}{cc}
0 & 0 \\
0 & 1 \\
\end{array}\Big]$, and $P=\Big[\begin{array}{cc}
1 & 0 \\
0 & 1 \\
\end{array}\Big]$. To show Lyapunov stability of (\ref{2nd_n1}) and (\ref{2nd_n2}), consider the Lyapunov-Razumikhin function given by
$V(x)=\frac{1}{2}\max_{i=1,2}\{(x_{i}-\alpha)^{2}\}$,
where $\alpha\in\mathbb{R}$. To use Razumikhin theorems, we focus on
$V(\phi(0))=\max_{-h\leq s\leq 0}V(\phi(s))$, that is, for the cases in which
\begin{eqnarray}\label{phiIN}
(\phi_{I}(0)-\alpha)^{2}\geq(\phi_{j}(s)-\alpha)^{2}, \quad s\in[-h,0],\quad j=1,2,
\end{eqnarray}
where $I$ is the index for which $|\phi_{I}(0)-\alpha|=\max_{i=1,2}|\phi_{i}(0)-\alpha|$.

The derivative of $V$ along the solutions of (\ref{2nd_n1}) and (\ref{2nd_n2}) is given by $\dot{V}(\phi(0))=(\phi_{I}(0)-\alpha)\dot{\phi}_{I}(0)$. First suppose $\phi_{I}(0)-\alpha=c>0$. Then it follows from (\ref{phiIN}) that $-c\leq\phi_{j}(s)-\alpha\leq c$ for all $j=1,2$ and $s\in[-h,0]$. This implies that $\phi_{j}(s)-\phi_{I}(0)\leq 0$ for all $j=1,2$ and $s\in[-h,0]$. Hence, by (\ref{2nd_n1}) and (\ref{2nd_n2}), it follows that $\dot{\phi}_{I}(0)\leq 0$. Similarly, one can show that if $\phi_{I}(0)-\alpha\leq 0$, then $\dot{\phi}_{I}(0)\geq 0$. In summary, $\dot{V}(\phi(0))\leq 0$. Now it follows from Theorem 4.1 of \cite[p.~152]{HV:1993} that (\ref{2nd_n1}) and (\ref{2nd_n2}) is Lyapunov stable relative to $\mathcal{C}$. Now it follows from Theorem~\ref{coro1b} that the time-delay system given by (\ref{2nd_n1}) and (\ref{2nd_n2}) is semistable relative to $\mathcal{C}$. 
\end{example}

\section{Conclusion}

A new framework concerning semistability of nonlinear systems having nonisolated equilibria and bounded time-varying delays is presented and its applications to stability analysis of multiagent dynamic networks with consensus protocol in the presence of unknown heterogeneous time-varying delays are discussed in this report. Those time delays are not necessarily differentiable and known. We provided conditions, in terms of the limiting delay system, to guarantee semistability of nonlinear systems with multiple time-varying delays and applied those stability results to show that multiagent dynamic networks can still achieve consensus in the presence of heterogeneous, autonomous self-regulating time lags.

\bibliographystyle{IEEEtran}
\bibliography{Reference}

\end{document}